\documentclass[12pt]{article}      
\usepackage{graphicx}

\usepackage{graphicx}
\usepackage{url}
\usepackage{color}
\definecolor{maroon}{rgb}{.69,.188,.376}
\definecolor{darkgreen}{rgb}{0,.5,0}
\definecolor{darkblue}{rgb}{0,0,.5}
\definecolor{magenta}{rgb}{1,0,1}

\usepackage[mathscr]{euscript}		

\usepackage{dsfont}

\usepackage{psfrag}			
\usepackage[colorlinks=true]{hyperref}
\hypersetup{pdftex, colorlinks=true, linkcolor=maroon, citecolor=maroon,
  filecolor=blue,urlcolor=blue}
\usepackage{amsmath}
\usepackage{amsthm}

\usepackage{amssymb}

\usepackage{caption} 

\usepackage{anysize}

\usepackage{enumerate} 
\usepackage{enumitem}

\usepackage[top=.95in, bottom = 1 in, left=1in, right = .6in]{geometry}

\renewcommand{\Re}{\mathrm{Re}\,}

\newtheorem{theorem}{Theorem}[section]
\newtheorem{lemma}[theorem]{Lemma}
\newtheorem{corollary}[theorem]{Corollary}
\newtheorem{proposition}[theorem]{Proposition}
\newtheorem{remark}[]{Remark}

\numberwithin{equation}{section}

\definecolor{Red}{rgb}{1,0,0}
\definecolor{Blue}{rgb}{0,0,1}
\definecolor{Olive}{rgb}{0.41,0.55,0.13}
\definecolor{Yarok}{rgb}{0,0.5,0}
\definecolor{Green}{rgb}{0,1,0}
\definecolor{MGreen}{rgb}{0,0.8,0}
\definecolor{DGreen}{rgb}{0,0.55,0}
\definecolor{Yellow}{rgb}{1,1,0}
\definecolor{Cyan}{rgb}{0,1,1}
\definecolor{Magenta}{rgb}{1,0,1}
\definecolor{Orange}{rgb}{1,.5,0}
\definecolor{Violet}{rgb}{.5,0,.5}
\definecolor{Purple}{rgb}{.75,0,.25}
\definecolor{Brown}{rgb}{.75,.5,.25}
\definecolor{Grey}{rgb}{.7,.7,.7}
\definecolor{Black}{rgb}{0,0,0}

\newcommand{\ignore}[1]{{}}

\date{\today}

\begin{document}

\title{Continuity of Zero-Hitting Times of Bessel Processes and Welding Homeomorphisms of SLE$_\kappa$}
\author{
Dmitry Beliaev
\thanks{University of Oxford. Email:\url{belyaev@maths.ox.ac.uk}}
\and Atul Shekhar%
  \thanks{University Lyon 1.
    Email: \url{atulshekhar83@gmail.com}}
\and Vlad Margarint
\thanks{NYU Shanghai.
    Email: \url{margarint@nyu.edu}}
 }


\maketitle
\begin{abstract}
We consider a family of Bessel Processes that depend on the starting point $x$ and dimension $\delta$, but are driven by the same Brownian motion. Our main result is that almost surely the first time a process hits $0$ is jointly continuous in $x$ and $\delta$, provided $\delta\le 0$. As an application, we show that the SLE($\kappa$) welding homeomorphism is continuous in $\kappa$ for $\kappa\in [0,4]$. Our motivation behind this is to study the well known problem of the continuity of SLE$_\kappa$ in $\kappa$. The main tool in our proofs is random walks with increments distributed as infinite mean Inverse-Gamma laws. 
\end{abstract}

\noindent {\em Keywords:} Bessel Processes, Schramm-Loewner-Evolution, Welding Homeomorphism.

\noindent {\em AMS 2010 Subject Classification:} 30C55, 60J65, 60H10.

\section{Introduction} \label{intro}

In this article we prove the joint continuity of level zero hitting times of Bessel processes w.r.t. its starting point and its dimension. For a real $\delta$, the Bessel process of dimension $\delta$ started from $x \in \mathbb{R}\setminus\{0\}$ is defined as the solution to the stochastic differential equation (SDE) 
\begin{equation}\label{Bessel-eqn}
d Z_{t}=dB_{t}+\frac{\delta-1}{2} \frac{1}{Z_{t}} d t, \hspace{2mm} Z_{0}=x,    
\end{equation}
where $B_t$ is a standard Brownian motion. Let $\zeta_\delta^x := \inf\{t> 0 | Z_t(x) = 0\}$. Also, set $\zeta_\delta^0 =0$. It is well known that $\zeta_\delta^x < \infty$  almost surely if and only if $\delta <2$. For a fixed starting point $x$, the random variable $\zeta_\delta^x$ is very well understood. There is an extensive literature covering the subject, see e.g. \cite{lawler2018notes}. We are interested in $\zeta= \{\zeta_{\delta}^x\}_{x,\delta }$ considered as a stochastic process indexed by $x$ and $\delta$. Our main result is the following theorem: 

\begin{theorem}\label{main-result-1}
The function $(x, \delta)\mapsto \zeta_\delta^x$ is almost surely jointly continuous in $x\in \mathbb{R}$ and $\delta \leq 0$. 
\end{theorem}

\begin{remark}
The continuity of $\zeta_\delta^x$ w.r.t. $x$ is not expected to hold for $\delta > 0$. For example, when $\delta =1$, $\zeta_1^x$ is a L\'evy subordinator process which in particular has jumps. However, the almost sure continuity of $\zeta_\delta$ at a fixed $x$ follows easily for all $\delta < 2$ using a Laplace transform computation, see \cite[Lemma 5]{altman2018bismut} for details. The continuity of $\zeta$ for $\delta\leq 0$ also implies that $x\mapsto \zeta_\delta^x$ is a continuous increasing bijection of $[0,\infty)$. For $\delta \in (0, 3/2]$, this function is injective, but not surjective (or equivalently continuous). For $\delta \in (3/2,2)$, this will not be injective with positive probability, see \cite[Proposition 2.11]{lawler2018notes}. \\

\end{remark}

The process $\zeta$ is very closely related to Schramm-Loewner-Evolutions (SLEs). We provide an application of Theorem \ref{main-result-1} to the continuity in $\kappa$ for the welding homeomorphism of SLE$_\kappa$ for $\kappa \in [0,4]$. Let us first recall some definitions and mention our initial motivation to consider this problem. 

\vspace{2mm}

Let $\mathbb{H}= \{x+ iy \hspace{1mm}|\hspace{1mm} y>0 \}$ be the upper half plane. Given a simple curve $\gamma:[0, T] \rightarrow \mathbb{H} \cup\{0\}$ such that $\gamma_{0}=0$
and $\gamma_{t} \in \mathbb{H} \hspace{2mm}$ for all $t>0$, the welding homeomorphism associated to $\gamma$ is defined as follows. Let $f: \mathbb{H} \rightarrow \mathbb{H} \setminus \gamma[0, T]$ be the (unique) conformal map such that $\lim _{z \rightarrow 0} f(z)=\gamma_{T}$
and $f(z)=z+O(1)$ as $z \rightarrow \infty .$ The map $f$ extends continuously to $\overline{\mathbb{H}}$ (see Chapter $2$ in \cite{pommerenke2013boundary}). For some real numbers $x_{T}^{-}< 0 <x_{T}^{+}$, 
$f$ maps both $\left[x_{T}^{-}, 0\right]$ and $\left[0, x_{T}^{+}\right]$ to $\gamma[0, T].$
The intervals $\left(-\infty, x_{T}^{-}\right]$ and $\left[x_{T}^{+}, +\infty \right)$ are similarly mapped under $f$ to $(-\infty,0]$ and $[0, \infty)$ respectively. The welding homeomorphism $\phi=\phi_{\gamma}:[0, \infty) \rightarrow[0, \infty)$
associated to $\gamma$ is defined by the relation $f(x)^{2}=f(-\phi(x))^{2},$ i.e. for $x \in\left[0, x_T^{+}\right]$, $\phi(x)$ is the unique point such that $f(x)=f\left(-\phi(x)\right)$, and for $ x \in \left[x_{T}^{+}, \infty \right),$ $\phi(x)$ is the unique point such that $f(-\phi(x))=-f(x).$ The homeomorphism $\phi$ contains information about the curve $\gamma$. For example, when $\phi$ is quasisymmetric, it uniquely characterizes $\gamma$, see \cite{lehto-book}. \\

For $\kappa \in [0,4]$, it was proven in \cite{schramm2005basic} that SLE$_{\kappa}$ is almost surely a simple curve, call it $\gamma^\kappa$. We will write $\phi^\kappa$ for the associated welding homeomorphism. We ask ourselves whether these homeomorphisms $\phi^\kappa$ are continuous in $\kappa$. Our motivation to ask this is to study the related problem of continuity of $\gamma^\kappa$ in $\kappa$. To best of our knowledge it is an open problem for the full range of $\kappa \in [0,\infty)$ or even for $\kappa\in [0,4]$, see \cite{viklund2014continuity} for a result for $\kappa \in [0, 8(2-\sqrt{3})) \cup (8(2+\sqrt{3}),\infty) $ and \cite{friz2019regularity} for a recent progress for $\kappa < 8/3$. Our approach to this problem is based on the following heuristic argument. \\
It follows from the results of \cite{schramm2005basic} and \cite{jones2000removability} that SLE$_{\kappa}$, for $\kappa \in [0,4)$, are almost surely conformally removable. This implies that $\phi^\kappa$ almost surely characterize the curve $\gamma^\kappa$ uniquely. In other words, the homeomorphism $\phi^\kappa$ contain all the information about the curve $\gamma^\kappa$. Heuristically speaking, this suggests that the continuity of $\phi^\kappa$ in $\kappa$ should imply the continuity of $\gamma^\kappa$ in $\kappa$ for $\kappa < 4$. Note however that this roadmap is as of now incomplete. This is because $\phi^\kappa$ are not quasisymmetric (otherwise this would imply that $\gamma^k$ is a quasislit, and then a result of Rohde-Marshall \cite{marshall2005loewner} would imply that Loewner driving function of $\gamma^k$, which is $\sqrt{k}B$, is $1/2$-H\"older). It is interesting to ask for fine properties of $\phi^\kappa$ which are satisfied uniformly in $\kappa$ and which recovers $\gamma^\kappa$ uniquely. We plan to address this in our future projects. For the purpose of present article we only prove the continuity of $\phi^\kappa$ in $\kappa$. \\

Asking for the continuity of $\phi^\kappa$ in $\kappa$ is not yet well posed if we work with the above definition of $\phi^\kappa$. This is because it is a priori not known whether $\gamma^\kappa$ are curves (let alone simple curves) simultaneously for all $\kappa \in [0,4]$ (we will often say that that a collection of events $\left\{A_{\alpha}\right\}_{\alpha}$ occur simultaneously in $\alpha$ if $\mathbb{P}[\cap_{\alpha} A_\alpha] =1$). This indeed is itself very closely related to the continuity of $\gamma^\kappa$ in $\kappa$, which is the problem we want to address in the first place. The correct way to formulate this problem is to ask for a continuous modification of the stochastic field $\{\phi^\kappa(x)\}_{x\geq 0, \kappa\in [0,4]}$. Our following theorem answers it.

\begin{theorem} \label{main-result-2} There exists a random field $\psi(\kappa, x):[0,4] \times[0, \infty) \rightarrow[0, \infty)$ such that
\begin{enumerate}
\item Almost surely, $\psi$ is jointly continuous in $(\kappa, x) \in [0,4] \times[0, \infty)$.
\item Simultaneously for all $\kappa \in [0, 4], \psi(\kappa, \cdot)$ is a homeomorphism of $[0, \infty)$. 
\item  $\mathbb{P}\left[\phi^{\kappa}=\psi(\kappa,\cdot)\right]=1$, $\hspace{2mm} \forall \kappa \in [0, 4].$
\end{enumerate}
\end{theorem}

\vspace{2mm}

\begin{remark}
We believe that there is an alternative approach to Theorem \ref{main-result-2} based on Sheffield's Quantuam Zipper. It was proven in \cite{Sheffield_QZ} that welding homeomorphism can be constructed by identifying points with same quantum length. This is also a promising approach, but it does require some additional work to give a rigorous proof. For example, we will need continuity of the  quantum measure $\mu^\gamma$ with respect to the quantum parameter $\gamma$. For $\gamma <2$ (corresponds to $\kappa<4$), this was done in \cite{Jun}.  For $\kappa =4$ or $\gamma =2$ these measures converge to $0$, so one has to consider an appropriate scaling limit (see \cite{APS}). Another issue is that we need all measures $\mu^\gamma$ to be `nice' simultaneously for all $\gamma$, so that we can invert the map $x\mapsto \mu^\gamma([0,x])$ simultaneously for all $\gamma$. All this requires some additional work. We believe that this could be done, but this approach is highly technical for proving the above Theorem which is relatively simple. We thus give a self contained proof of this result using the simpler approach based on Bessel processes.    
\end{remark}

The paper is organized as follows. In the Section \ref{prem}, we recall some basic facts on Loewner theory and Bessel processes. Some technical lemmas are proved in Section \ref{technical}. In the Section \ref{main-proofs} we give the construction of function $\psi$ using an intermediary result Proposition \ref{Prop 1}, and prove Theorem \ref{main-result-1} and Theorem \ref{main-result-2}. Finally, we prove the Proposition \ref{Prop 1} in the Section \ref{prop-proof}.  \\

\vspace{2mm}

\textbf{Acknowledgments:} D.B. was supported by the Engineering
\& Physical Sciences Research Council (EPSRC) Fellowship EP/M002896/1. A.S.  acknowledges the financial support from the European Research Council (ERC) through a project grant LIKO. V.M.  acknowledges the support of NYU-ECNU Institute of Mathematical Sciences at NYU Shanghai.

\section{Preliminaries}\label{prem}

We recall some basic facts from the Loewner theory. Given a curve $\gamma$ as described in the introduction, one can choose a parametrization
of $\gamma$ such that $\forall t \geqslant 0$, the half plane capacity
of $\gamma[0, t]$ is
$2t$, i.e. $\lim _{z \rightarrow \infty} z\left(g_{t}(z)-z\right)=2 t$, where $g_{t}: \mathbb{H} \backslash \gamma[0, t] \rightarrow \mathbb{H}$ is the unique conformal map such that $g_{t}(z)-z \rightarrow 0 $, as $|z| \rightarrow \infty$. We will assume that $\gamma$ is defined for $t \in[0,1]$ in this parametrization. The Loewner transform of $\gamma$ is the real-valued continuous function $U$ defined by $U_{t}:=\lim _{z \rightarrow \gamma_{t} ; z \in \mathbb{H}\setminus\gamma[0,t]} g_{t}(z)$. For each $z \in \mathbb{H}$, $g_t(z)$ satisfies the ordinary differential equation (ODE) given by 

\begin{equation}\label{LDE}
\partial_tg_t(z)=\frac{2}{g_t(z)-U_t}, \hspace{2mm} g_0(z)=z.
\end{equation}
We refer to the equation \eqref{LDE} as the Loewner differential equation (LDE). \\
This process could be reversed. Given a driving function $U_t$ one can solve LDE \eqref{LDE}. The resulting map $g_t$ is a conformal map from the set of points $\mathbb{H}\setminus K_t$ where the solution exists up to time $t$ onto the upper half-plane. It is a standard fact that for  $U=\sqrt{\kappa} W$ where $\kappa\in [0,4]$ and $W$ is a standard Brownian motion, there is a continuous simple curve $\gamma=\gamma^{\kappa}$ such that $K_t=\gamma[0,t]$.  The curves  $\gamma^{\kappa}$ are known as SLE$_{\kappa}$ (curves). From now on we assume that the driving function $U_t$ is of this form. \\

To recover the curve $\gamma$ (when it exists) from $U$, it is beneficial to look at the flow associated to reverse-time LDE as follows. Let $\hat{U}_{t}=U_{1}-U_{1-t}$ be the time reversal of $U$. For each fixed $s \in [0,1]$, $t \geq s$ and $z \in \mathbb{H}$, let $h(s,t,z)$ denote the solution of the reverse time stochastic LDE given by 
\begin{equation}
\label{reverse SDE}
d h(s, t, z)= d\hat{U}_{t}-\frac{2}{h(s,t,z)} dt, \quad h(s, s, z)=z \in \mathbb{H}.
\end{equation}

The map $h:\{0 \leqslant s \leqslant t \leqslant 1\} \times \mathbb{H} \rightarrow \mathbb{H}$ is called
the flow associated with the equation \eqref{reverse SDE} and
it satisfies the so called flow property:
\[
 h(s, t, z)=h(u,t, h(s, u,z)),\qquad \forall \ s \leqslant u \leqslant t.
\] 

We will need the following Lemma from \cite{rohde2018loewner}.
\begin{lemma}\label{up-bound}
If $z=iy$, $y > 0$, then 
\[
|\Re h(s, t, z))|\leqslant 2 \sup_{r \in[s, t]}| \hat{U}_{r}-\hat{U}_{s}|,
\]
and
\[
\operatorname{Im}(h(s,t,z)) \leq \sqrt{y^2+4(t-s)}.
\]
\end{lemma}
\vspace{5mm}
The following lemma is a rewriting of Lemma $2.1$ from \cite{shekhar2017remarks} and it follows easily from \eqref{LDE}.
\begin{lemma}\label{unif-maps}
If $f_{t}:\mathbb{H}\rightarrow \mathbb{H}\setminus \gamma [0,t]$ is the conformal map such that $\lim _{z \rightarrow 0} f_t(z)=\gamma_{t}$ and $f_{t}(z)=z+O(1) \text { as } |z| \rightarrow \infty$, then $f_{t}(z)=h(1-t, 1, z).$ In particular
$f_{1}(z)=h(0,1, z)$
\end{lemma}
\vspace{4mm}

The welding homeomorphism of a simple $\gamma[0,1]$ defined in the Section \ref{intro} can thus be constructed using the continuous extension of $h(0,1,\cdot)$ to $\overline{\mathbb{H}}$. It is therefore natural to consider solution $h(s,t,x)$ of \eqref{reverse SDE} started from $x \in \mathbb{R}\setminus \{0\}$. Note however that in this case the solution might hit zero in finite time and we will consider $h(s,t,x)$ only up to this hitting time.

When $U=\sqrt{\kappa} W,$ we will denote the time reverse Brownian motion $\hat{W}$ by $B$ and write $h^{\kappa}(s, t, z)$ for the flow obtained when $\hat{U}=\sqrt{\kappa}B$. Note that if $Z^\delta(s,t,x)$ denote the solution to \eqref{Bessel-eqn} with the initial value $Z^\delta(s,s,x) =x$, then for $\kappa \neq 0$ and $\delta=1-\frac{4}{\kappa}$ ($\delta$ and $\kappa$ are henceforth always related as such),
\begin{equation}\label{loewner=bessel}
\frac{h^{\kappa}(s, t,\sqrt{\kappa}x)}{\sqrt{\kappa}} = Z^\delta(s,t,x).    
\end{equation}
It follows that $ T_{\kappa}(s, \sqrt{\kappa}x) = \zeta_{\delta}(s, x)$, where
\[
\begin{aligned}
T_\kappa(s,x)&:=\inf\left\{t>s | h^{k}\left(s, t, x\right)=0\right\}
\\
\zeta_\delta(s,x)&:=\inf\left\{t>s | Z^\delta\left(s, t, x\right)=0\right\}.
\end{aligned}
\] 
Also set $T_{\kappa}(s,0)= \zeta_{\delta}(s, 0)=s$. To simplify some notations we will use $T_\kappa(x)$ and $\zeta_\delta(x)$ to denote $T_\kappa(0,x)$ and $\zeta_\delta(0,x)$.

The following result in well known, see e.g. Proposition $2.9$ and Proposition $2.11$ of \cite{lawler2018notes}. Recall the Inverse-Gamma$(\alpha,\beta)$ distribution has a density proportional to $t^{-1-\alpha}\exp\{-\beta/t\}$.  

\begin{lemma}\label{exact-law}

\begin{enumerate}
\item For $\delta<2,$ the $\zeta_\delta(1)$ has the Inverse-Gamma$(1-\frac{\delta}{2},\frac{1}{2})$ law. 
\item If $\delta \leq \frac{3}{2}$, then for all $0<x<y<\infty$,
\[
\mathbb{P}[\zeta_\delta(x)< \zeta_\delta(y)]=1.
\]
    
\end{enumerate}

\end{lemma}

\section{Some Technical Lemmas}\label{technical}
In this section we prove some technical lemmas we will need to prove  our main results.

\begin{lemma}[Gronwall inequality]\label{gronwall}
Let $x>0$ and $M_t$ and $N_t$ satisfy $$M_{t} \leqslant x+ \hat{U}_{t}-\int_{0}^{t} \frac{2}{M_{r}} dr \hspace{2mm} (\mbox{respectively} \geq)$$ 
and $$N_{t}=x+\hat{U}_{t}-\int_{0}^{t} \frac{2}{N_{r}} d r.$$ 
Then $M_t \leq N_t$ (resp. $M_t \geq N_t$). In particular, for $z=x+iy$, $x,y >0$, 
\begin{equation}\label{z>x}
\Re(h^\kappa(0,t,z)) \geq h^\kappa(0,t,x) \hspace{2mm}\mbox{for all} \hspace{2mm} t\leq T_\kappa(x).
\end{equation}
\end{lemma}
\begin{proof}
Note that $$M_{t}-N_{t} \leqslant \int_{0}^{t} \frac{2\left(M_{r}-N_{r}\right)}{M_{r} N_{r}} d r,$$ 
and the claim follows from Gronwall inequality. In particular, if $h^{\kappa}(0,t,z)= X_t+ iY_t$, then 
\[
d X_{t}=\sqrt{\kappa} dB_{t}-\frac{2 X_{t}}{X_{t}^{2}+Y_{t}^{2}} d t \geq \sqrt{k} d B_{t}-\frac{2}{X_{t}} d t,
\]
which implies \eqref{z>x}.
  
\end{proof}

\begin{lemma}\label{12}
As $\kappa \to 0+$, $\sqrt{\kappa}\sup_{t \leq T_{\kappa}(x)}|B_t|\to 0$ uniformly over $x$ in compact sets. 
\end{lemma}

\begin{proof}Using monotonicity and scaling property of $T_\kappa(x)$ w.r.t. $x$, it suffices to consider $x=1$. Note that $T_{\kappa}(x)= \zeta_{\delta}(1 / \sqrt{\kappa})$. Let $\kappa_n=2^{-n}$ and $\kappa_{n+1}\leq \kappa\leq \kappa_n,$ then 
\[
\zeta_{\delta}\left(\frac{1}{\sqrt{\kappa}}\right) \leqslant \zeta_{\delta}\left(\frac{1}{\sqrt{\kappa_{n+1}}}\right) \leqslant \zeta_{\delta_{n}}\left( \frac{1}{\sqrt{\kappa_{n+1}}}\right).
\]
So,  
\[
\sqrt{\kappa} \sup_{t \leqslant T_{\kappa}(1)}|B_{t}|\leqslant \sqrt{\kappa_{n}} \sup_{t \leqslant \zeta_{\delta_{n}}\left( {1}/{\sqrt{\kappa_{n+1}}}\right)} |B_{t}|.
\]
Now, using Chebyshev inequality and Burkholder-Davis-Gundy inequality we obtain that
\[
\begin{aligned}
\mathbb{P}\left(\sqrt{\kappa_n}\sup_{t \leqslant \zeta_{\delta_{n}}( {1}/{\sqrt{\kappa_{n+1}}})} |B_{t}| \geq \kappa_n^{1/4} \right)
& 
\leq \frac{{\kappa_n}\mathbb{E}\left(\sup_{t \leqslant \zeta_{\delta_{n}}(1/\sqrt{\kappa_{n+1}})} |B_{t}|\right)^2}{\sqrt{\kappa_n}}
\\
&=O\left(\sqrt{\kappa_{n}} \mathbb{E}\left[ \zeta_{\delta_{n}}\left(\frac{1}{\sqrt{\kappa_{n+1}}}\right)\right]\right).
\end{aligned}
\]
Note that 
$\zeta_{\delta_{n}}(1) \sim \text{Inverse-Gamma}\left(\frac{2}{\kappa_{n}}+\frac{1}{2}, \frac{1}{2}\right)$, hence
\[
\mathbb{E}\left[\zeta_{\delta_{n}}(1)\right]=\frac{1}{2}\frac{1}{\frac{2}{\kappa_{n}}-\frac{1}{2}}=O\left(\kappa_{n}\right),
\]
which implies 
\[
\mathbb{P}\left(\sqrt{\kappa_n}\sup_{t \leqslant \zeta_{\delta_{n}}(  1/\sqrt{\kappa_{n+1}})} |B_{t}| \geq \kappa_n^{1/4} \right)=O(\sqrt{\kappa_n}).
\]
Borel-Cantelli Lemma implies that for $n$ large enough, 
\[
\sqrt{\kappa_n}\sup_{t \leqslant \zeta_{\delta_{n}}(  1/\sqrt{\kappa_{n+1}})}  |B_{t}| \leq \kappa_n^{1/4},
\] 
and the conclusion follows. 
\end{proof}

We will also use the following lemma on random walks with Inverse-Gamma$(1,\frac{1}{2})$ increments. Let $\{T_n\}_{n \geq 1}$ be an i.i.d. sequence of random variables each distributed as Inverse-Gamma$\left(1, 1/2 \right)$.
Note that $\mathbb{E}[T_1]=+\infty$ and the strong law of large numbers implies that almost surely $\frac{T_{1}+\ldots+T_{n}}{n} \longrightarrow+\infty$, i.e. $T_1+T_2+\dots+T_n$ tends to infinity faster than linear function. The following lemma gives that the precise speed of convergence is $n \log n$. The additional $\log n$ factor will be crucial for our proofs.   

\begin{lemma}\label{Lemma 5}
If $\{S_n\}_{n \geq 1}$ is a sequence of random variables such that $\forall n \geq 1$, $S_{n} \stackrel{d}{=} T_{1}+\cdots+T_{n}$, then 
\[
\frac{S_{n}}{n \log n} \stackrel{p}{\longrightarrow} \frac{1}{2}.
\]
\end{lemma}

\begin{proof}
We show that the Laplace transforms 
\begin{equation}\label{(3)}
\mathbb{E}\left[\exp \left(\frac{-tS_n}{n \log n}\right)\right]\rightarrow e^{-t/2},
\end{equation}
 as $n \to \infty.$ Then, the L\'evy continuity Theorem implies the claim. To prove \eqref{(3)}, note that 
\begin{equation}
\label{laplace}
\begin{aligned}
\mathbb{E}\left[\exp \left(\frac{-t S_{n}}{n \log n}\right)\right]&=\left(\mathbb{E}\left[\exp\left(\frac{-t T_{1}}{n \log n}\right)\right]\right)^{n}\\
&=\left(\sqrt{\frac{2 t}{n \log n}} K_{1}\left(\sqrt{\frac{2 t}{n \log n}}\right)\right)^{n},
\end{aligned}
\end{equation}
where $K_1(x)$ is the modified Bessel function of the second kind. We have used the fact that $\mathbb{E}\left[e^{-t T_{1}}\right]=\sqrt{2 t} K_{1}(\sqrt{2 t}).$ Finally, note that 
$$\lim _{x \rightarrow 0^{+}} \frac{\log \left(x K_{1}(x)\right)}{x^{2}(\log x+1)}=\frac{1}{2},$$ 
and plugging this asymptotics in \eqref{laplace} gives \eqref{(3)}.

\end{proof}

\section{Construction of the field $\psi(\kappa,x)$.} \label{main-proofs}

In this section we give the construction of $\psi(\kappa,x)$. This will be based on the following Proposition. 

\begin{proposition}\label{Prop 1}
\begin{enumerate}

\item Almost surely for all $\kappa \in [0,4]$ simultaneously, the function $x \longmapsto T_{\kappa}(x)$ is a strictly increasing continuous bijection $[0, \infty) \rightarrow[0, \infty)$, and it is a strictly decreasing continuous bijection $( -\infty, 0] \rightarrow[0, \infty).$

\item Almost surely for all $\kappa \in [0,4]$ simultaneously, the function $\kappa \mapsto T_{\kappa}(\cdot) \in C([0, \infty),[0, \infty))$ or $C((-\infty, 0],[0, \infty)]$ is continuous.

\end{enumerate}
\end{proposition}

The proof of this proposition is postponed until the next Section.

Proposition \ref{Prop 1} has a simple corollary. To state it we will need the following notations.  
We assume that  $\kappa \in [0,4]$ and $t \in[0,1]$ . 
For $x \geqslant 0$ we define
\[
\tilde{h}_{t}^{\kappa, +}(x)=\left\{\begin{array}{l}
h^{\kappa}(0,t,x), \quad \text { if } t \leq T_{\kappa}(x) \\
T_{\kappa}(x)-t, \text { if } t \geqslant T_{\kappa}(x)
\end{array}\right.
\]
Similarly, for  $x \leq 0$ we define
\[
\tilde{h}_{t}^{\kappa,-}(x)=\left\{\begin{array}{c} h^{\kappa}(0,t,x), \text { if } t \leq T_{\kappa}(x) \\ t-T_{\kappa}(x), \text { if } t \geqslant T_{\kappa}(x)\end{array}\right.
\]
The definition of the $\tilde{h}_{t}^{\kappa, \pm}$ is a bit artificial for $t>T_{\kappa}(x)$, but it will help us represent the welding homeomorphisms in a neat way. 
An immediate corollary to Proposition \ref{Prop 1} is the following.
\begin{corollary}
\begin{enumerate}

\item Almost surely for all $\kappa \in [0,4]$ simultaneously, maps $x \longmapsto \tilde{h}_{1}^{\kappa, +}(x)$ and $x \longmapsto \tilde{h}_{1}^{\kappa, -}(x)$ are strictly increasing continuous bijections $[0, \infty) \rightarrow[-1, \infty)$ and $(-\infty, 0] \rightarrow(-\infty, 1]$ respectively. 
\item Furthermore, the functions
$$\kappa \mapsto \tilde{h}_{1}^{\kappa,+} \in C([0,\infty),[-1, \infty)), $$
$$\kappa \mapsto \tilde{h}_{1}^{\kappa,-} \in C((-\infty, 0],(-\infty, 1])$$
are continuous.  

\end{enumerate}
\end{corollary}

We now define the continuous field $\psi$. Set $\psi(\kappa,0) =0$. For $x \in (0, \infty)$, let $\psi^{\kappa}(x)$ be the unique point such that 

$$\tilde{h}_{1}^{\kappa,-}\left(-\psi^{\kappa}(x)\right)=-\tilde{h}_{1}^{\kappa,+}(x).$$
Note that this definition is designed so that $h^{\kappa}(0,t, \cdot)$ started at $x$ and $-\psi^{\kappa}(x)$ either hit zero at the same time or $h^{\kappa}(0,1, x)=-h^{k}\left(0,1,-\psi^{\kappa}(x)\right)$. This is consistent with the definition of $\phi$ given in the Section \ref{intro}.
The proof of Theorem \ref{main-result-1} and Theorem \ref{main-result-2}-$(a),(b)$ is immediate from Proposition \ref{Prop 1}.

\begin{proof}[Proof of Theorem \ref{main-result-2}-(c)]
To prove $\phi^{\kappa}=\psi(\kappa,\cdot)$, using Lemma \ref{unif-maps}, it suffices to verify that 
\begin{equation}\label{phi=psi}\lim _{z \rightarrow x} h^{\kappa}(0,1, z)^{2}=\lim _{z \rightarrow-\psi^{\kappa}(x)} h^{\kappa}(0,1, z)^{2}.
\end{equation}
If $T_\kappa(x) >1$, then by definition, $T_\kappa(-\psi(\kappa,x))>1$. This implies that 
\[ h^\kappa(0,1,z)^2 \to h^\kappa(0,1,x)^2 \hspace{2mm} \mbox{as}\hspace{2mm} z \to x,\]
and
\[h^\kappa(0,1,z)^2 \to h^\kappa(0,1,-\psi(\kappa,x))^2 \hspace{2mm} \mbox{as}\hspace{2mm} z\to \psi(\kappa,x).\]
Then, the \eqref{phi=psi} follows by the definition of $\psi(\kappa,x)$.

If $T_\kappa(x) \leq 1$, then $T_\kappa(-\psi(\kappa,x))= T_\kappa(x) \leq 1$. Let $T_\kappa(-\psi(\kappa,x))= T_\kappa(x) = 1-t_0$, $t_0\geq 0$. Using the flow property, $h^\kappa(0,1,z) = h^\kappa(1-t_0,1,h^\kappa(0,1-t_0,z))$. We claim that 
\begin{equation}\label{goes-to-zero}
    h^\kappa(0,1-t_0,z) \to 0 \hspace{2mm}\mbox{as} \hspace{2mm} z \to x.
\end{equation} 
Then, using Lemma \ref{unif-maps}, it follows that $h^\kappa(0,1,z) \to \gamma_{t_0}^\kappa$ as $z\to x$. Similarly, $h^\kappa(0,1,z) \to \gamma_{t_0}^\kappa$ as $z\to -\psi(\kappa, x)$ as well, establishing \eqref{phi=psi}. 

To prove \eqref{goes-to-zero}, note that as $z\to x$, $\Re(z)$ is arbitrarily close to $x$. Then, using Lemma \ref{gronwall} and the continuity of $T_\kappa(x)$ in $x$, it follows that $\Re(h^\kappa(0,t,z)) > 0$ for all $t \leq T_\kappa(x)-\epsilon(z)$, where $\epsilon(z) \to 0$ as $z\to x$. Then it easily follows that $h^\kappa(0,T_\kappa(x)-\epsilon(z),z))$ is arbitrarily small. Finally, the \eqref{goes-to-zero} follows from the Lemma \ref{up-bound}. 

\end{proof}

\section{Proof of Proposition \ref{Prop 1}.} \label{prop-proof}

\begin{proof}[Proof of Proposition \ref{Prop 1}-(a)]

We first claim that almost surely simultaneously for all $\kappa \in [0,4]$ and all $s\in [0,1]$ (or equivalently for all $s \geq 0$ by a scaling argument), 
\begin{equation}\label{4}
\lim _{x \rightarrow 0} T_{\kappa}(s, x)-s=0.
\end{equation}
When $\kappa=0$, it follows from an explicit computation that 
$h^{0}(s, t, x)=\sqrt{x^{2}-4(t-s)},$ which implies $T_0(s,x)-s=\frac{x^2}{4}$ and \eqref{4} easily follows. For $\kappa \in (0,4],$  it suffices to consider $x \to 0+$. Using \eqref{loewner=bessel} and monotonicity of Bessel processes w.r.t. its dimension, it follows that if $0 <\kappa_1<\kappa_2 \leq 4$, then  
\[
 T_{\kappa_{1}}(s, \sqrt{\kappa_{1}} x)-s \leq T_{\kappa_{2}}(s, \sqrt{\kappa_{2}}, x)-s \leqslant T_{4}(s,2 x)-s.
\]
It suffices to prove that almost surely for all $s \in [0,1]$, 
\[
\lim _{x \rightarrow 0+} T_{4}(s, x)-s=0.
\]
Note that $T_k(s,x)$ is monotonic increasing in $x$ and the limit 
\[
 T_4(s, 0+)-s := \lim _{x \rightarrow 0+} T_{4}(s, x)-s
\]
 always exists. We now prove that this limit is zero for all $s \in [0,1]$. To this end, let 
\[x_{n}=\sqrt{\frac{3 e^{-n^3}}{n^{3}}}, \hspace{4mm} k_{n}=e^{n^{3}}, \hspace{4mm} \lambda_n=\frac{1}{n}.\]

\begin{figure}[h] 
\centering
\includegraphics[scale=0.3]{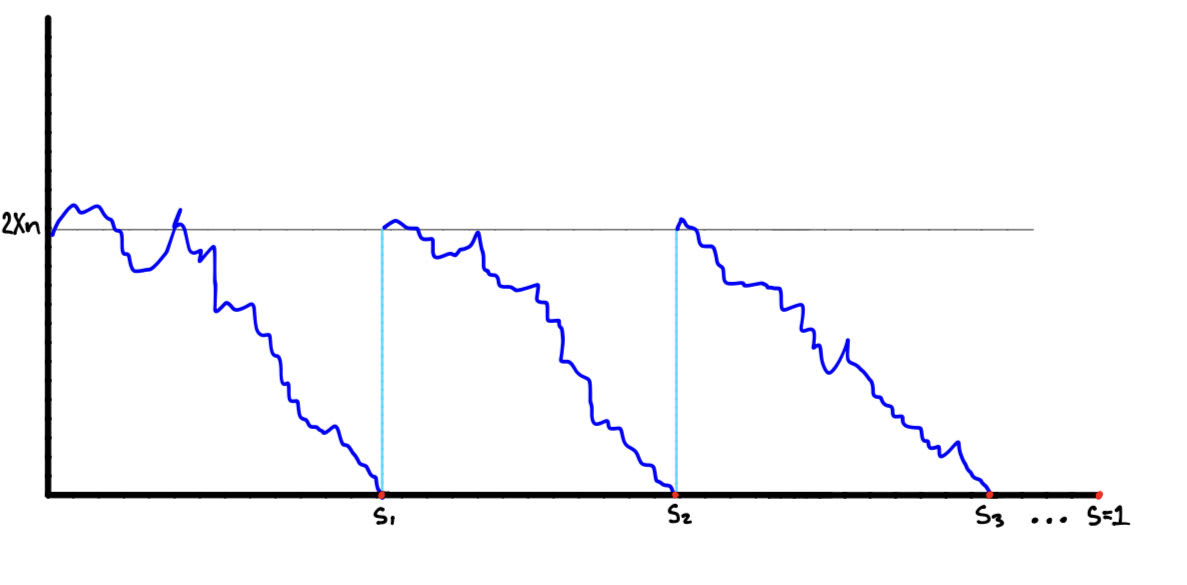}
\caption{Random walk construction of zero-hitting times.}
\label{fig:RW}
\end{figure}

For each $n \geq 1$, define a sequence $\{s^n_k\}_{k\geq 0}$ by $s_0^n=0$, and $s_{k+1}^n=T_4(s_k^n, 2x_n)$, see Figure \ref{fig:RW}. Note that by scaling, Strong Markov Property of the Brownian motion and Lemma \ref{exact-law}, $\{s_k^n\}_{k \geq 0}$ is a random walk with the increments distributed according to $x_n^2\times\text{Inverse-Gamma}(1,1/2)$. Then, Lemma \ref{Lemma 5} implies that  
$$\frac{s^n_{k_n}}{x_n^2 k_n\log k_n} \stackrel{p}{\longrightarrow} \frac{1}{2}.$$
Note that $x_n^2 k_n \log k_n>2$, and since convergence in probability implies almost sure convergence along a subsequence, we obtain that, almost surely, \begin{equation}\label{(7)}
    s^n_{k_n} >1 \hspace{3mm} \mbox{for infinitely many} \hspace{2mm} n.
\end{equation} 

Next, consider the event 
$$ A_{n}:=\bigcup_{k=0}^{k_{n}-1}\left\{s_{k+1}^{n}-s_{k}^{n}>\lambda_{n}\right\}.
$$
Then, using independence and the fact that 
\[
\mathbb{P}\left[\text { \text{Inverse-Gamma}}\left(1, \frac{1}{2}\right) \leq \lambda\right]=\exp \left(\frac{-1}{2 \lambda}\right),
\] 
we get
\[
\mathbb{P}\left[A_{n}\right]=1-\exp \left(\frac{-k_{n} x_{n}^{2}}{2 \lambda_{n}}\right) \leqslant \frac{k_{n} x_{n}^{2}}{2 \lambda_{n}}.
\]
Note that 
\[
\sum_{n=1}^{\infty} \frac{k_{n} x_{n}^{2}}{\lambda_{n}}<\infty,
\] 
and the Borel-Cantelli Lemma implies, almost surely, for all $n$ large enough 
$$
s_{k+1}^{n}-s_{k}^{n} \leqslant \lambda_{n}, \hspace{2mm}  \forall k=0,1, \ldots k_{n}-1.
$$

Now, for any $s \in [0,1]$, using \eqref{(7)}, we can find infinitely many $n$ such that for some $0 \leq k\leq k_n-1$, $s \in [s_k^n, s_{k+1}^n]$. 
Using the flow property and the monotonicity, we obtain that
\[
T_{4}(s,0+)-s \leqslant T_{4}\left(s, h^{4}\left(s_{k}^{n}, s,2 x_{n}\right)\right)-s=
T_{4}(s_{k}^{n}, 2 x_{n})-s \leqslant s_{k+1}^{n}-s_{k}^{n} \leqslant \lambda_{n},
\]
which implies that 
\[
T_{4}(s,0+)-s=0.
\] 

\vspace{6mm}

The fact that $x \mapsto T_\kappa( x)$ on $[0, \infty)$ is strictly increasing follows easily from Lemma \ref{exact-law}-(b). As for its continuity, we first prove the left continuity. For any $x \in (0, \infty),$ if $y \uparrow x$, let $T_{\kappa}(x-):=\lim _{y \uparrow x} T_{\kappa}(y)$. 
If $T_{\kappa}(x-)<T_{\kappa}( x)$, then by taking the monotonic limit of $h^{\kappa}(0,t,y)$ as $y \uparrow x$, we obtain a solution to the \eqref{reverse SDE} starting from $x$ which hits zero before time $T_{\kappa}(x)$. Since \eqref{reverse SDE} has a unique solution, this gives a contradiction. Thus, $T_{\kappa}(x-)=T{\kappa}(x)$. 

For the right-continuity of $T_{\kappa}(x)$, for any $0 \leq x<y <\infty$, using again the flow property, we have that that 
\[
T_{\kappa}(y)-T_{\kappa}( x)=T_{\kappa}(T_{\kappa}(x), h^{\kappa}\left(0, T_{\kappa}(x), y\right))-T_{\kappa}(x).
\]
Also, as $ y \downarrow x$, a similar monotonicity argument as above implies that $h^{\kappa}\left(0, T_{\kappa}(x), y\right) \rightarrow 0$. Thus, \eqref{4} implies that $\lim _{y \rightarrow x+} T_{\kappa}(y)=T_{\kappa}(x)$, finishing the proof.

\end{proof}
 
\begin{proof}[Proof of Proposition \ref{Prop 1}-(b)]
We first check the continuity in $\kappa$ at $\kappa=\kappa_0 \in (0,4]$. Since $T_{\kappa}( x)= \zeta_{\delta}( x / \sqrt{\kappa})$, it suffices to check the continuity of $\zeta_{\delta}(\cdot)$ in $\delta$. Note that if $0 <\kappa_1<\kappa_2 \leq 4$, then  $\zeta_{\delta_{1}}(x)\leq \zeta_{\delta_{2}}(x)$. If either $\kappa\downarrow \kappa_0$ or $\kappa \uparrow \kappa_0$, we will establish the pointwise convergence 
$\zeta_{\delta}(x)\rightarrow \zeta_{\delta_{0}}(x)$. Then, by Dini's Theorem, we obtain the uniform convergence on compact sets. 

For pointwise convergence, let $\kappa \uparrow \kappa_0$ first. Note that $Z^{\delta}(0, t, x)$ is monotonically increasing with $\delta$ (or $\kappa$). If $\lim _{\kappa \uparrow \kappa_{0}} \zeta_{\delta}(x)<\zeta_{\delta_{0}}(x)$, then by taking the limit function $\lim _{\kappa \uparrow \kappa_{0}} Z^{\delta}(0, t, x),$ one can construct a solution to \eqref{Bessel-eqn} with $\delta=1-\frac{4}{\kappa_0}$ started from $x$ which hits zero before $\zeta_{\delta_0}(x)$, which is a contradiction. 

For $\kappa \downarrow \kappa_0$, using the flow property, 
\[
 \zeta_{\delta}(x)-\zeta_{\delta_{0}}(x) =\zeta_{\delta}(\zeta_{\delta_{0}}(x), Z^{\delta}\left(0, \zeta_{\delta_{0}}(x), x\right))-\zeta_{\delta_{0}}(x).
\]

Again, using a similar argument as before, it is easy to check that $Z^{\delta}\left(0, \zeta_{\delta_0}(x), x\right) \rightarrow 0$ as $\kappa \downarrow  \kappa_0$. Using \eqref{4} again implies $\lim _{\kappa \downarrow \kappa_{0}} \zeta_{\delta}(x)=\zeta_{\delta_0}(x)$.

The continuity in $\kappa$ at $\kappa =0$ requires a different argument. Note that for $t \leq T_{\kappa}(x)$, 
\[
h^{\kappa}(0,t,x)=x+\sqrt{\kappa} B_{t}-\int_{0}^{t} \frac{2}{h^{\kappa}(0,r,x)} d r \leq x+\sqrt{\kappa} \sup_{t \leq T_{\kappa}(x)}B_t -\int_{0}^{t} \frac{2}{h^{\kappa}(0,r,x)} d r. 
\]
Then, Lemma \ref{gronwall} implies that 
\[
h^{\kappa}(0,t,x) \leq \sqrt{(x+\sqrt{\kappa} \sup_{t \leq T_{\kappa}(x)}B_t)^2-4t}.
\]
Thus, 
\[
T_{\kappa}(x) \leq \frac{(x+\sqrt{\kappa}\sup_{t \leq T_{\kappa}(x)}B_t)^2}{4}.
\]
Similarly, 
\[
T_{\kappa}(x) \geq \frac{(x+\sqrt{\kappa}\inf_{t \leq T_{\kappa}(x)}B_t)^2}{4}.
\]
Using Lemma \ref{12}, we conclude that $T_{\kappa}(x) \rightarrow \frac{x^2}{4}$ uniformly on compact sets as $\kappa \to 0+$, completing the proof. 
\end{proof}


\bibliographystyle{plain}

\end{document}